\documentclass[12pt, reqno]{amsart}
\usepackage{geometry}
\usepackage{amsthm}
\usepackage{amssymb}
\newtheorem{definition}{Definition}[section]
\newtheorem{theorem}{Theorem}[section]

\newtheorem{conjecture}{Conjecture}[section]
\usepackage{enumerate}
\usepackage{graphicx}
\numberwithin{equation}{section}
\usepackage{hyperref}

\title[Properties of Polynomial and Transcendental Fractals]{Invariance and Inner Fractals in Polynomial and Transcendental Fractals}
\author{Nabarun Mondal}
\address{D.E.Shaw \& Co. India, Hyderabad }
\email{mondal@deshaw.com}
\thanks{ }

\author{Partha P. Ghosh}
\address{Microsoft India, Hyderabad }
\email{parthag@microsoft.com}
\thanks{Dedicated to Benoit B. Mandelbrot. \\ 
Dedicated to our parents and children, without their presence we are nothing. \\
Big thanks to Dhrubajyoti Ghosh, Arindam Chanda, for their been constant support. 
}

\begin{document}


\begin{abstract}
A lot of formal and informal recreational study took place 
in the fields of Meromorphic Maps,
since Mandelbrot popularized the map $z \leftarrow z^2 + c$.
An immediate generalization of the Mandelbrot  $z \leftarrow z^n + c$ 
also known as the Multibrot family were also studied.  
In the current paper, general truncated polynomial maps 
of the form $z \leftarrow \sum_{p \ge 2}{a_p x^p}  + c$ are studied. 
Two fundamental properties of these polynomial maps are hereby presented. 
One of them is the existence of shape preserving transformations on fractal images,
and another one is the existence of embedded Multibrot fractals inside a polynomial fractal.
Any transform expression with transcendental terms also shows embedded Multibrot fractals,
due to Taylor series expansion possible on the transcendental functions.
We present a method by which existence of embedded fractals can be predicted. 
A gallery of images is presented alongside to showcase the findings.
 
\end{abstract}

\maketitle

\begin{section}{Introduction}
Mandelbrot popularized the map $T(z) : z \leftarrow z^2 + c$ which bears his name (\cite{cfnfs},\cite{cmns}). 
The Mandelbrot set $M$, comprise of points `$z$', such that $T(z)$, repeated infinite time, 
upon itself, remains bounded. 

The iteration is defined as in \cite{cfnfs} :-
\begin{enumerate}

\item{Initialize current point $c=P(x,y)$ and $z_0=c$. }
\item{Use the Transform map $T : z \leftarrow z^2 + c$ to generate $z_1$.}
\item{Repeat the procedure (2), to get $z_2,z_3,...,z_l$ until $|z_l| > k$ or $l=N$ is the maximum  number of iteration.}
\item{If iterations expired, the original point $c \in M$.}
\item{Else $c \not \in M$.}
\item{Take another point in $(x,y)$ plane, and start again from step (1), till all the points in the plane are exhausted.}

\end{enumerate}

When $T(z)$ is applied on itself once it can be said $T^2(z)$. 
Making things general, 
an `$x$' times repetition of $T(z)$ upon itself would be termed as $T^x(z)$.

Then, Formally:-
\begin{equation}\label{mandelbrot}
M(k) = \{ z : \lim_{x\to\infty} |T^x(z)| \le k \}
\end{equation}

A plot of the set yields outstanding visual imagery, and patterns.

\begin{subsection}{Escape Time Fractals}
A generic way to produce any set $\Phi$ of this kind would be then to have a transform function $F(z) : z \leftarrow F(z)$ such as to:-

\begin{equation}\label{generic_fractal_k}
\Phi(k) = \{ z : \lim_{x\to\infty} |F^x(z)| \le k \}
\end{equation}

Clearly then, $k$ is a very important parameter.
However the removal of $k$ can be easily done with:- 

\begin{equation}\label{generic_fractal}
\Phi =   \{ z : \lim_{x\to\infty} |F^x(z)| < \infty  \}
\end{equation}

For in case of Mandelbrot set $M$ it can be easily shown that any point $|z| > 2 $ would escape to infinity, that means:-
$$
M(z^2 + c,2) = M(z^2 + c,k) \; ; \; 2 < k \le \infty  
$$
so for $M$ , k is actually omitted. 

The images generated by equation \eqref{generic_fractal} are called ``Escape Time Fractals'',
also known as ``orbits'' fractals.
\end{subsection}

\begin{subsection}{Our Methodology}
We have modified an open source `Mandelbrot Set Viewer' computer program,
and incorporated the ability to type in mathematical expressions for the function $F(z)$ in equation \eqref{generic_fractal}.
The source code is available \href{http://code.google.com/p/dynamicfractalviewer/}{here}.

There is an expression handler in the program that can handle 
many standard mathematical functions. 
Given the expression, the program would evaluate the expression
dynamically, to generate the `escape time fractal' for the function $F(z)$.
The program is capable of:-

\begin{enumerate}
\item { zoom-in and zoom-out. } 
\item { modifying the radii of convergence `$k$', to be expressed as a power of `$e$'. }
\item { modifying number of iterations to perform as `infinity'.}
\item { saving the image as `png' format image file, that is how illustrations of this paper were generated. } 
\item { plotting the set in Colors, to be discussed more in next.}
\end{enumerate}

\begin{subsubsection}{Plotting the Set using Software}
From a computational point of view, there is no way to perform the iteration of equation \eqref{generic_fractal} up-to infinity.
Hence, a maximum `$N$' number of iteration is performed. Using the custom software we can manipulate `$N$', $N \in [0,99999]$.
If a point $z$ exceeds the limit $k$ at iteration  number $m$, then a ratio $x = m/N$ is defined. 
Clearly $x=1$ means, the function remained bounded till the `$N$'th iteration.
Based upon the value of $x$, the color of the point is chosen, among a fixed set of colors. 
Though a point might show $x=1.0$, that does not imply that, it would 
remain so for an increased $N$. Also, another point having current $x < 0$ might get into the set, for an increased $k$.
The coloring can be formally defined as a mapping:- 

\begin{equation}\label{coloring}
C(x) : \mathbb{R} \rightarrow \mathbb{N} 
\end{equation} 

The custom application which produced the illustrations of present paper is capable of plotting in any integer number of colors, 
and can handle any  mathematical expression and function.
The color codes are dependent on the color map used.
The program uses equation \eqref{coloring_program} for coloring.

\begin{equation}\label{coloring_program}
C(x) :  \lceil { p \; x } \rceil \; ; \; \forall x \in (0,1)
\end{equation} 
In equation \eqref{coloring_program} `$p$' is the total number of colors in the color map. 
The default value for color maps used is $p=256$, and the default color is Gray Scale. 
The user of the software can choose from a list of color maps, or can create his own color map.  
\end{subsubsection}

\begin{subsubsection}{Unique Findings from the Experiments}
As the software can evaluate almost any mathematical expressions to generate the fractals of type equation \eqref{generic_fractal}, 
we noticed certain interesting findings.

The most interesting phenomenon we noticed is the existence of `embedded' fractals while using Transcendental expressions, that is fractals, 
which are created by expressions with transcendental functions in it. The first observation was `$tan(z)^2 + c$' fractal 
( figure  ~\ref{fig:3_7} on page ~\pageref{fig:3_7} )contains `$z^2+c$' or Mandelbrot fractals inside it 
(figure ~\ref{fig:3_8} on page ~\pageref{fig:3_8}).

There are Mandelbrot fractals inside `${(tan(z))}^{-2} + c$' fractal. 
Figure  ~\ref{fig:3_9} on page ~\pageref{fig:3_9} illustrates a ${(cotan(z))}^2+c$ fractal, 
and figure ~\ref{fig:3_10} on page ~\pageref{fig:3_10} depicts the embedded Mandelbrot fractals inside it.

The second observation was, that the Mandelbrot set is capable of `Shape Preserving Transformations',
like rotation, scaling, and translation, just by multiplying the $z^2$ term by different (complex) factors. 
We have established a `generic' way to transform polynomial fractals without a $z^1$ term.

\end{subsubsection}

\end{subsection}

\begin{subsection}{Generalization of the Transformation Function}

Equation \eqref{mandelbrot} can be further generalized into:-

\begin{equation}\label{multibrot}
\mathbb{F} : z \leftarrow z^n + c \; ; \; n \in \mathbb{N} 
\end{equation}

where $\mathbb{N}$ is the set of Integers.

Further generalization can be done by:-

\begin{equation}\label{gen_multibrot}
\mathbb{F} : z \leftarrow z^n + c \; ; \; n \in \mathbb{R} 
\end{equation}

where $\mathbb{R}$ is the set of real numbers.
The resulting sets are known as ``Multibrots'' family.
In particular, \cite{mult_n} and \cite{mult_p} consider these maps in details.

\begin{figure}
\begin{center}
\leavevmode
\includegraphics[scale=0.3]{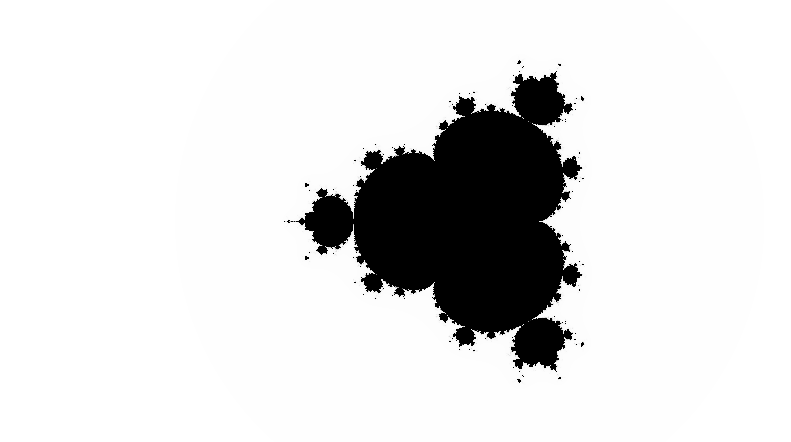}
\end{center}
\caption{Multibrot, with $n=4$ }
\label{fig:1_1}
\end{figure}

Figure  ~\ref{fig:1_1} on page ~\pageref{fig:1_1} shows a ``Multibrot'' fractal with $n=4.0$.

\end{subsection}

\end{section}

\begin{section}{Invariance : Shape Preserving Transforms on Fractal Shapes}

We know that the only shape preserving transforms are:-
\begin{enumerate}

\item{Translation}
\item{Rotation}
\item{Scaling}

\end{enumerate} 

In this section we review the integer multibrot map \eqref{multibrot}, 
under the influence of shape preserving transformations.
Our observation is that \emph{Standard shape preserving transformations works on 
the function family of `Multibrots' with positive `n'} (equation \eqref{multibrot} , $n>0$).
We note that the family with $n \in \mathbb{Q}$ does not preserve shape under rotation, 
simply because the transform cease to become a function :-
$$
z \leftarrow z^\frac{a}{b} + c \implies z \leftarrow  {^b}\sqrt{z^a} + c
$$
The expression  ${^b}\sqrt{z^a}$ is not a function, and therefore no 1-1 relationship 
between the old point, and mapped point can be established.

\begin{subsection}{Shape Preserving Transformation Theorems }

We show that the point which is now current transformed `$^tz$' , has the same trajectory
of the earlier `$z$'. So, if `$z$' was included in the set, so would be `$^tz$' with modification in $k$'s value,
and, if  `$z$' was not included in the set, so would not be `$^tz$'. That would establish the transformation properties.

\begin{theorem}\label{trans} 
\textbf{A variable transformation of the form $z : z - a$ translates a Multibrot Map.}

Translation can be achieved by :-
\begin{equation}\label{translation}
T_t(z) : z \leftarrow (z-a)^n + c \; ; \; a(x,y) \in \mathbb{C} \; ; \; n>0
\end{equation}
The result is the set image shifted co-ordinates to $a = (x,y)$.
\end{theorem}

One point to be noted here is that as the radii of convergence is $k$, 
if the point $a$ resides further than $k$, then we need to increase the radii
appropriately, for seeing the whole set's image.

\begin{proof}[Proof of the Translation Theorem]

To prove this we note that all points which were mapped as $z$ according to equation \eqref{multibrot}, 
are now mapped as $z+a$ in the new mapping with equation \eqref{translation}.
That would mean that if $z \in M$ for equation \eqref{multibrot}, then $z+a \in M'$ for equation \eqref{translation},
and vice versa. We take a point $^tc = c + a$ and hence $^tz_0 = c + a$.

Clearly at 1st level the original transformation can be written as:-

\begin{equation}\label{z_1}
z_1 = z^n_0 + c = c^n + c
\end{equation}

Compare with 1st level transformation using equation \eqref{translation}:-
$$
{^tz}_1 = ({^tz}_0 -a)^n + {^tc} = ({^tc} - a )^n + {^tc} = (c + a - a)^n + c + a = c^n + c + a 
$$

where ${^tz}$ and ${^tc}$ signifies the transformed $z,c$. 

So we have established that:- 
\begin{equation}\label{z_1_t}
{^tz}_1 =z_1 + a
\end{equation} 

Now, let it be true for some $l \ge 1$.
Then we have:-
$$
{^tz}_l =z_l + a
$$ 

We have:-
\begin{equation}\label{z_l_1}
z_{l+1} = z^n_l + c 
\end{equation}

and
$$
{^tz}_{l+1} = ({^tz}_l - a)^n + {^tc} 
$$
Replacing the value of ${^tz}_l$ we have:-
$$
^tz_{l+1} = (z_l + a  - a)^n + ^tc = (z^n_l + c) + a 
$$

hence:-
\begin{equation}\label{z_l_1_t}
^tz_{l+1} =  z_{l+1} + a
\end{equation}

Using equation \eqref{z_1_t} as basis and then having equation \eqref{z_l_1_t},
implies that using principle of mathematical induction, we can say
all points which were mapped as $z$ according to equation \eqref{multibrot}, 
are now mapped as $z+a$ in the new mapping with equation \eqref{translation}. 
That establishes the translational property.
\end{proof}

\begin{theorem}\label{rot}
\textbf{ Transformation $z : e^{i\theta}z$ rotates a Multibrot Map.}

A ``clockwise'' rotation of angle $\rho = \theta/(n-1)$ can be achieved by:-
\begin{equation}\label{rotation}
T_r(z) : z \leftarrow e^{i\theta}z^n + c 
\end{equation}
\end{theorem}

\begin{proof}[Proof of the Rotation Theorem]
To prove that we note again:-
$z_0 = c$ and $^tz_0 = ce^{(\frac{-i\theta}{n-1})}$.

We have:-
$$
{^tz}_1 = e^{i\theta}({^tz}^n_0) + ^tc = c^ne^{(\frac{ in\theta - i\theta -in\theta}{n-1})} + ce^{(\frac{-i\theta}{n-1})} = e^{(\frac{-i\theta}{n-1})} ( c^n + c) 
$$

Hence, using equation \eqref{z_1}
\begin{equation}\label{z_1_r}
{^tz}_1 = e^{(\frac{-i\theta}{n-1})}z_1 
\end{equation} 

Assume now that for some $l \ge 1$

$$
^tz_l = e^{(\frac{-i\theta}{n-1})}z_l 
$$

Clearly we would have:-

$$
{^tz}_{l+1} = e^{i\theta}({^tz}^n_l) + ^tc =    e^{(\frac{-i\theta}{n-1})} z^n_l +  ce^{(\frac{-i\theta}{n-1})}  = e^{(\frac{-i\theta}{n-1})} ( z^n_l + c)
$$

Then, using equation \eqref{z_l_1} we get:-
\begin{equation}\label{z_l_1_r}
{^tz}_{l+1} = e^{(\frac{-i\theta}{n-1})} z_{l+1} 
\end{equation}

Using equation \eqref{z_1_r} as basis and having equation \eqref{z_l_1_r}
implies that, using principle of mathematical induction, we can say
all points which were mapped as $z$ according to equation \eqref{multibrot}, 
are now mapped as $e^{(\frac{-i\theta}{n-1})}z$ in the new mapping with equation \eqref{rotation}. 
That establishes the rotational property.

\end{proof}

Note that for ``Multibrots'' with $n<0$ this formula holds true.
As the term $\rho < 0 \forall n < 0 $ , this rotation would be in ``anti-clock-wise'' direction.
For $n>0$ , this rotation is in ``clockwise'' direction.
Also note that, for ``Multibrots''with fractional $n>0$, rotation leads to altogether different type of
images. They are not invariant under rotation.

\begin{theorem}\label{scale} 
\textbf{Transformation  $z : az $ scales a Multibrot Map}

A Scaling of factor $\sigma$ can be achieved by :-
\begin{equation}\label{scaling}
T_s(z) : z \leftarrow az^n + c \; ; \; \sigma = a^{(\frac{1}{1-n})} \; ; \; a \in {\mathbb{R}}_{+} \; ; \; n > 0 
\end{equation}
The resulting figure is a scaled version of the original Map, by a factor of $\sigma$.
\end{theorem}

\begin{proof}[Proof of the Scaling Theorem]

To show this, we start with  $z_0 = c$ and 
$^tz_0 = ca^{(\frac{1}{1-n})}$ as initial points.

Clearly then:-

$$
{^tz}_1 = a({^tz_0}^n) + ^tc = a( c a^{(\frac{1}{1-n})} )^n +   ca^{(\frac{1}{1-n})} =  a^{(\frac{1}{1-n})} ( c^n + c)
$$
hence, using equation \eqref{z_1}

\begin{equation}\label{z_1_s}
^tz_1 = a^{(\frac{1}{1-n})} z_1
\end{equation}

Now, assume it is true for some $l \ge 1$

$$
{^tz}_l =  a^{(\frac{1}{1-n})} z_l
$$

Clearly we can say:-

$$
{^tz}_{l+1} = a({^tz_l}^n) + ^tc =  a ( a^{(\frac{1}{1-n})} z_l )^n + ca^{(\frac{1}{1-n})} = a^{(\frac{1}{1-n})} ( z^n_l + c )
$$

Then, using equation \eqref{z_l_1} we get:-
\begin{equation}\label{z_l_1_s}
{^tz}_{l+1} = a^{(\frac{1}{1-n})} z_{l+1} 
\end{equation}

Using equation \eqref{z_1_s} as basis and having equation \eqref{z_l_1_s}
implies that, using principle of mathematical induction, we can say
all points which were mapped as $z$ according to equation \eqref{multibrot}, 
are now mapped as $a^{(\frac{1}{1-n})}z$ in the new mapping with equation \eqref{scaling}. 
That establishes the scaling property.

\end{proof}
   
Note that, ``Multibrots''with fractional $n>0$ are  invariant under scaling.  

\end{subsection}

\begin{subsection}{Generic Shape Preserving Transformation}

The most generic shape preserving transformation would be, then
\begin{equation}\label{spt}
T(z) : z \leftarrow a(z-b)^n + c \; ; \; a=|a|e^{i\theta},b(x,y) \in \mathbb{C}
\end{equation}

The transformation of equation \eqref{spt} translates the origin by $b=(x,y)$ amount, 
then rotates the map by $\theta/(n-1)$, and then scales down the map by $|a|^{(\frac{1}{1-n})}$.

\begin{subsubsection}{Shape Preserving Transformations For Polynomial Fractals}

Polynomials are important due to Taylors Series expansion available for 
any infinitely differentiable function.
The generic formula for a polynomial transform is:-
\begin{equation}\label{poly_transform}
T(z) = \sum\limits_{j=0}^n a_j z^j  + c = P(z) + c
\end{equation}
We note that  while equation \eqref{translation} can be immediately applied to a polynomial, 
due to variations in power transforms like rotation: equation \eqref{rotation} and scaling: equation \eqref{scaling} 
can not be used on equation \eqref{poly_transform}.
Hence, only with the translation part, we can redefine equation \eqref{spt} for polynomials as:-

\begin{equation}\label{t_poly}
T(z) : z \leftarrow  P(z-b)  + c \; ; \; b(x,y) \in \mathbb{C}
\end{equation}

\end{subsubsection}

\begin{subsubsection}{Rotation and Scaling of Truncated Polynomial Fractals}

We can not consider polynomial with a linear $z^1$ term because then the results 
from rotation \eqref{rotation} and scaling \eqref{scaling} becomes undefined with $n=1$. 
\begin{definition}\label{trunc_poly-def}
\textbf{Left Truncated Polynomial.}

A Left Truncated polynomial is defined as:-
\begin{equation}\label{trunc_poly}
P(z,l) = \sum_{k = l} a_k z^k \; ; \; l \ge 2   
\end{equation}

\end{definition}
We define a truncated polynomial is a polynomial with at least no linear term,
(and possibly constant term) to be defined formally in a later section, with equation \eqref{trunc_poly}.
We can drop the requirement of not having a constant term, because in that case $z^0$ , and the equations \eqref{rotation}
and \eqref{scaling} remains defined.

We have found out from equations \eqref{rotation} and \eqref{scaling} 
that they affect different powers of `z' differently. 
But assume, we still want to find a way to rotate, and scale a polynomial fractal image appropriately.
For notational convenience, we treat $P(z)$ as a vector, with components  $\chi_j = a_jz^j \; ; \; j \ne 1 $.

$$
\mathbf{P} = [ a_jz^j ]
$$

We note that, to generate an uniform scaling, $\sigma_i \ne \sigma_j$.

\begin{equation}\label{scale_vector}
\mathbf{s} = [ \sigma_j ] 
\end{equation}

We note that, to generate an uniform rotation, $\rho_i \ne \rho_j$.

\begin{equation}\label{rotation_vector}
\mathbf{r} = [\rho_j]
\end{equation}

It is easy to see that the for the `effect' of them all should be equal to some predefined value `$u$',
$$
u = {\sigma_j}^{\frac{1}{1-j}}    
$$ 
rewriting, we get:-

$$
\sigma_j = {u}^{1-j}
$$

In the same way, to get a fixed angle of rotation `$\theta$' :-

$$
\rho_j = e^{i(j-1)\theta}
$$

Hence, a properly scaled and rotated original Polynomial fractal $\mathbf{P} + c$ transform, 
with translation, would have components:-

\begin{equation}\label{spt_poly}
^{\sigma\rho t}\mathbf{P} = [^T\chi_j] = [ u^{1-j} a_j e^{i(j-1)\theta} (z-b)^j ]
\end{equation} 

Using equation \eqref{spt_poly} we can modify and shift any `non-infinite' ($j \ne \infty$ ) truncated ($j \ne 1$ ) polynomial fractal, 
such a way that the original shape is preserved. So, the set image remains `invariant' under equation \eqref{spt_poly}.

Note that equation \eqref{spt_poly} can handle $n<0$, but not the scaling part, as `$k$' needed to be adjusted.

\end{subsubsection}

\begin{subsubsection}{Demonstration of Invariance}
We take the equation of :
\begin{equation}\label{demo_inv}
T(z) = z^2 + z^3 + c 
\end{equation}
The image of equation \eqref{demo_inv} can be seen in figure  ~\ref{fig:3.x.1} on page ~\pageref{fig:3.x.1}.
The image of this can be rotated clockwise $\pi/2$ radian, by multiplying the $\chi_2$ by $i$ and 
$\chi_3$ by $e^{i\pi}$. Hence, 

\begin{equation}\label{demo_inv_res_1}
^{1,\frac{\pi}{2},1}T(z) = i{(z-1)}^2 - {(z-1)}^3 + c 
\end{equation}
Equation \eqref{demo_inv_res_1} is a clockwise $\pi/2$ rotated, and right shifted by 1, image of the equation \eqref{demo_inv},
which is clearly established by the figure ~\ref{fig:3.inv.1} on page ~\pageref{fig:3.inv.1}.
To further establish the result, we present the second equation:-

\begin{equation}\label{demo_inv_res_2}
^{1,\frac{\pi}{4},1}T(z) = e^{(i\frac{\pi}{4})}{(z-1)}^2 + i{(z-1)}^3 + c 
\end{equation}
Equation \eqref{demo_inv_res_2} is a clockwise $\pi/4$ rotated, and right shifted by 1, image of the equation \eqref{demo_inv},
which is clearly established by the figure ~\ref{fig:3.inv.2} on page ~\pageref{fig:3.inv.2}.
\end{subsubsection}
\end{subsection}

\begin{subsection}{Self Similarity}
With the previous set of transforms in mind, we focus on what is
observed self similarity in the Multibrot fractals.

The tool we seek is the magnify a small section of the fractal shape,
and finding how it should look. We note that this has to deal with 
the transform scaling ( theorem \ref{scale} ). But that won't be all.

The point with respect to which we need to \emph{magnify} or \emph{zoom}
first needs to be made the origin of the co-ordinate, and then \emph{magnified} , 
or \emph{scaled}. That would make the transform looks like 
\emph{translation followed by scaling}. But that is not so.

We established that the direction of the co-ordinate system matters for fractals,
as in Multibrot with fractional powers does not show remain same under an induced
co-ordinate rotation. Therefore, the transform should also take care of aligning 
the co ordinates properly, via rotation.

\begin{definition}\label{zoom}
\textbf{The Zoom Transform.}

Zoom of a fractal image can be done using :-

\begin{equation}\label{zoom-t}
z \leftarrow s e^{i\phi} f(z - a) + c  
\end{equation}
where $s$ is a positive constant factor, $a=(x_n, y_n)$ a complex number 
to have the co-ordinate origin shifted to, and $\phi$ denotes the rotation angle.
 
\end{definition}

\begin{conjecture}\label{pseudo-self-similarity}
\textbf{Existence of Self Similar Embedded Fractals.}

A fractal of the type :-
$$
z \leftarrow  f(z) + c
$$
has quasi self similar embedded fractals 
under magnification iff it is invariant under 
the equation \eqref{zoom-t} or the zoom transform.
\end{conjecture}

\end{subsection}

\end{section}

\begin{section}{Generic Analytic Family and Embedded Fractals}

It is to be noted here that all geometric transcendental 
functions can be expanded in Taylor series, and hence a transform of the form:

$$
T: z \leftarrow G(z) + c
$$ 

where $G(z)$ is a geometric transcendental function,
qualifies as an infinite polynomial map.

We have noticed an important property of such maps, that they have embedded Multibrot
fractals inside small regions within them. Some of them we have already shared in the introduction section.
We furnish some more examples here:- 
\begin{enumerate}
\item{The fractal of `$cos(z)-1 + c$' (figure ~\ref{fig:3_1} at page  ~\pageref{fig:3_1} ) 
contains Mandelbrot fractals inside (figure ~\ref{fig:3_2} at page  ~\pageref{fig:3_1} ).

\begin{figure}
\begin{center}
\leavevmode
\includegraphics[scale=0.3]{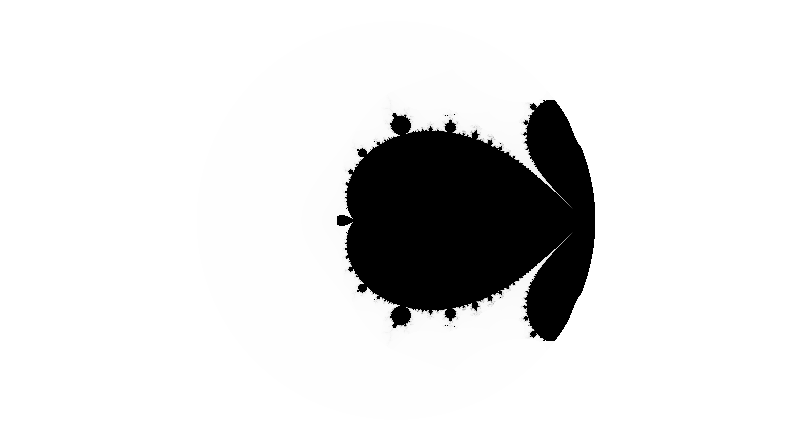}
\end{center}
\caption{ Fractal of $cos(z)-1+c$ , k is about 8 }
\label{fig:3_1}
\end{figure}

\begin{figure}
\begin{center}
\leavevmode
\includegraphics[scale=0.3]{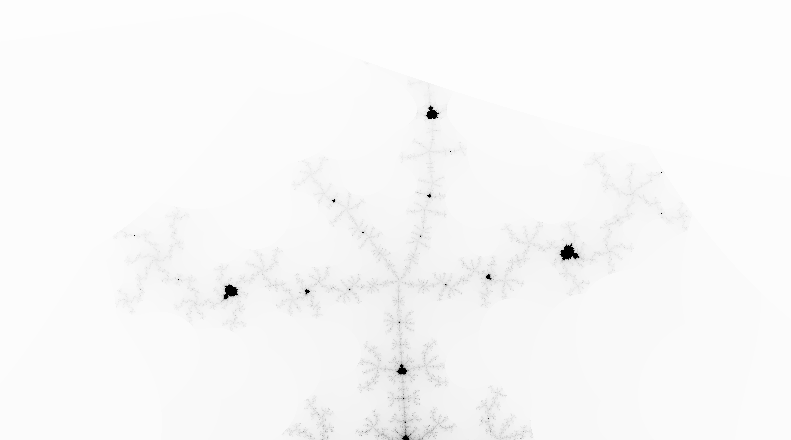}
\end{center}
\caption{ Fractal of $cos(z)-1+c$ showing Inner Mandelbrot, k is about 8 }
\label{fig:3_2}
\end{figure}
}

\item{The fractal of `$sin(z^2)+c$'  (figure ~\ref{fig:3_3} at page  ~\pageref{fig:3_3}) 
has Mandelbrots inside (figure ~\ref{fig:3_4} at page  ~\pageref{fig:3_4} ), and looks like a deformed eaten away Mandelbrot image overall.

\begin{figure}
\begin{center}
\leavevmode
\includegraphics[scale=0.3]{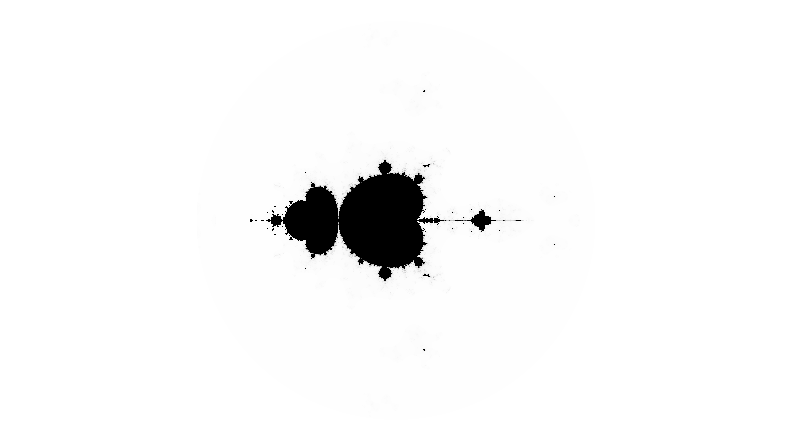}
\end{center}
\caption{ Fractal of $sin(z^2)+c$ , k is about 8 }
\label{fig:3_3}
\end{figure}

\begin{figure}
\begin{center}
\leavevmode
\includegraphics[scale=0.3]{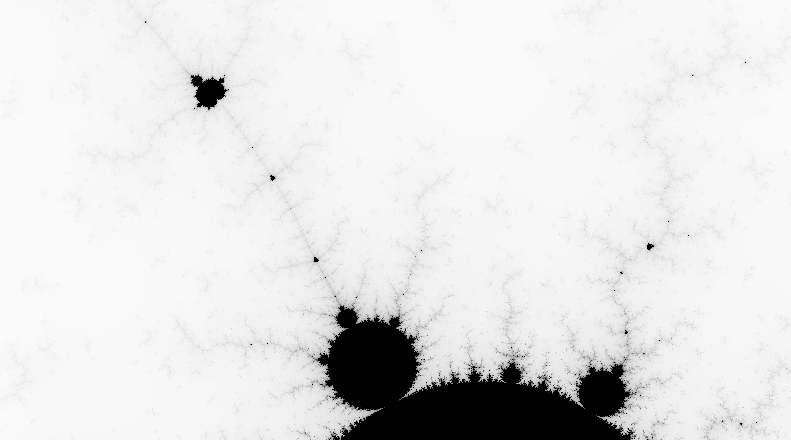}
\end{center}
\caption{ Fractal of $sin(z^2)+c$ showing Inner Mandelbrot, k is about 8 }
\label{fig:3_4}
\end{figure}

}

\item{Increasing the power,  `$sin(z^4)+c$' is (figure ~\ref{fig:3_5} at page  ~\pageref{fig:3_5}) virtually same as a Multibrot 
with $n=4$  (figure ~\ref{fig:1_1} at page  ~\pageref{fig:1_1}), and has Multibrots $n=4$ inside it (figure ~\ref{fig:3_6} at page  ~\pageref{fig:3_6}).

\begin{figure}
\begin{center}
\leavevmode
\includegraphics[scale=0.3]{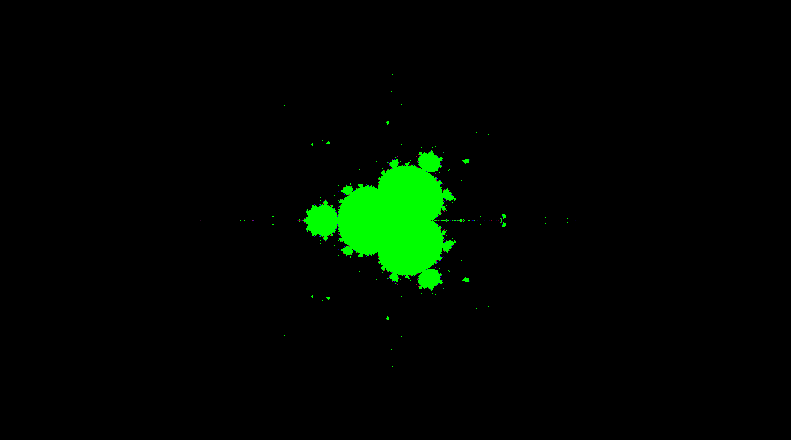}
\end{center}
\caption{ Fractal of $sin(z^4)+c$ , k is about 8 }
\label{fig:3_5}
\end{figure}

\begin{figure}
\begin{center}
\leavevmode
\includegraphics[scale=0.3]{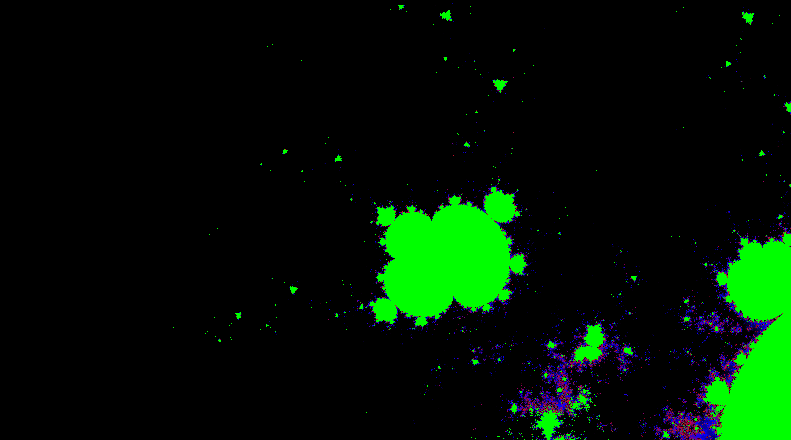}
\end{center}
\caption{ Fractal of $sin(z^2)+c$ showing Deformed Inner Multibrot with n=4 , k is about 8 }
\label{fig:3_6}
\end{figure}

}

\item{What more, the fractal of `$6(z - sin(z)) + c$' shows up just like a Multibrot of $n=3$ (figure ~\ref{fig:3_7} at page  ~\pageref{fig:3_7}). 

\begin{figure}
\begin{center}
\leavevmode
\includegraphics[scale=0.3]{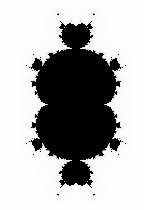}
\end{center}
\caption{ Fractal of $6(sin(z)-z)+c$ Resembles a Multibrot with n=3 , k is about 8 }
\label{fig:3_7}
\end{figure}

}

\end{enumerate}

These observations lead us to think that for those maps, the original transform function 
can probably be approximated by the Multibrot type function. 
This is the topic of discussion in this section.

\subsection{Asymptotic Bounds : Bachmann Landau notation}
What we informally seek is that in the neighborhood of a point, 
can we replace a function by another much simpler function, because, they behave approximately the same way.
This concept is already available as asymptotic bounds proposed by Bachmann \cite{pb}, and reintroduced by Landau \cite{el}, for real functions,
with a parameter $x$.  

\begin{definition}\label{big-theta}
\textbf{Big Theta : $\Theta$ . }

Let there exist functions $f(x) , g(x) $ and constants $k_1 > 0 , k_2 > 0$ such that :-
$$
k_1 f(x)  \le g(x) \le k_2 f(x) \; ; \;  x \to x_{crit}
$$
then, 
\begin{equation}\label{orig_theta}
g(x) \in \Theta(f(x))
\end{equation}
\end{definition}

The equation \eqref{orig_theta} is the Bachmann-Landau `$\Theta$' notation, 
mostly used in computer science (Knuth \cite{dk}) for getting approximate bounds for functions
(runtime analysis of algorithms).
With simple modification, complex analytic functions can be included in $\Theta$ notation.

\begin{definition}\label{big-theta-c}
\textbf{Complex Generalization of Big Theta. }

Iff for $F(z), G(z)$ and constants $k_1 , k_2  \in (0,k) $ 
at neighborhood of a point $z_c \in \mathbb{C}$:-
\begin{equation}\label{dom_rel}
k_1 |F(z)|  \le |G(z)| \le k_2 |F(z)| \; ; \; z \to z_c
\end{equation}
then, we define:-
\begin{equation}\label{dom_theta}
G(z_c) \in \Theta(F(z_c))
\end{equation}
\end{definition}

As both $F(z),G(z)$ are analytic, then, there are infinite points 
in the neighborhood of $z_c$ where equation \eqref{dom_rel} will be true.
On that neighborhood, $|F(z)|$ will approximate $|G(z)|$ with arbitrary accuracy.

\begin{theorem}\label{theta-poly}
\textbf{There is a  $\Theta$ relationship for Complex Polynomials. }

Let 
$$
f(z) = \sum_{k=l}^{u} {a_k z^k} \; ; \; a_k \in \mathbb{R}
$$ 
Then, if this holds:-
$$
\lim_{z \to z_c}{ \frac{|z|^k}{|z|^m} } = 0 
$$
then 
$$
\lim_{z \to z_c}{ f(z) \in \Theta (z^m ) }  
$$
In particular:-
$$
\forall |z| \to \infty  \implies  f(z) \in \Theta(  z^u ) 
$$
and,
$$
\forall |z| \to 0 \implies f(z) \in \Theta(  z^l ) 
$$
\end{theorem}
\begin{proof}[Proof of the theorem \ref{theta-poly} ]

We rewrite the polynomial as:-
$$
f(z) = z^m \sum_{k=l}^{u} { a_k \frac{z^k}{z^m} } 
$$
Taking mods we get:-

$$
|f(z)| = |z|^m | \sum_{k=l}^{u} { a_k \frac{|z|^k}{|z|^m} } | 
$$
which immediately gives the result.
\end{proof}

\begin{theorem}\label{theta-itr}
\textbf{The relation $\Theta$ is true over iterations over polynomials. }

Let $f(z)$ be a polynomial function with 
$$
\lim_{z \to z_c} f(z) \in \Theta(g(z)) ;
$$
where $g(z) = z^m $ Then,
$$
\lim_{z \to z_c} f^n(z) \in \Theta(g^n(z))
$$
\end{theorem}
\begin{proof}[Proof of the theorem \ref{theta-itr} ]
We rewrite the polynomial as:-
$$
f(z) = z^m \sum_{k=l}^{u} { a_k \frac{z^k}{z^m} } 
$$
and then the $f^n (z)$ becomes:-
$$
f^n(z) = z^{mn} \sum_{k=ln}^{un} { b_k \frac{z^k}{z^{mn}} } 
$$
Taking mod, the result follows immediately. 
 
\end{proof}

\begin{subsection}{The Principle of Dominance}

What we informally seek, is when we ``zoom in'' to a polynomial fractal, what ``inner'' fractal type we may find.
If we find a fractal of type $f_c(z) = f(z) + c $ at a point `$z_d$' inside a transformation function $T_c(z) = T(z) + c$, then we say that
the function $f_c(z)$ must have ``dominated'' the function $T_c(z)$ at point $z_d$.

\begin{theorem}\label{theta-fractal}
\textbf{Relationship $\Theta$ can be used to find escape time fractals.}

Let the transform function $z  \leftarrow T(z) $ be:-
$$
T = f(z) + c 
$$ 
$f(z)$ being a polynomial with $\lim_{z \to z_d} f(z) \in \Theta(g)$. 
Let then:-
$$
S = \{ z : \lim_{n \to \infty z \to z_d } { |f^n(z)| < \infty } \} 
$$
and
$$
S_\Theta = \{ z : \lim_{n \to \infty } { |g^n(z)| < \infty } \} 
$$
so that:-
$$
S = S_\Theta 
$$
\end{theorem}
\begin{proof}[Proof of the theorem \ref{theta-fractal}]

We note that from theorem \eqref{theta-itr} :-
$$
\lim_{z \to z_d } f^n \in \Theta( g^n ) 
$$
This would imply that:-
$$
|f^n| \le K_u |g^n|
$$
which would imply that:-
$$
S \subseteq S_\Theta 
$$
also 
$$
|f^n| \ge K_l |g^n|
$$
which would imply that:-
$$
S_\Theta \subseteq S  
$$
combining we get $S_\Theta = S$.
\end{proof}

\begin{theorem}\label{domth}
\textbf{Dominance : Relation $ f \in \Theta(g) $  guarantees a `sub-fractal'. }
 
Lets $\lim_{z \to z_d} f \in \Theta(g) $.
Then, there for the escape time fractal generated from the transform function:-
$$
z \leftarrow f(z) + c
$$ 
there will be a embedded sub fractal of the form:-
$$
z \leftarrow g(z) + c 
$$ 
In general for $f(z) = P(z,l)$ , 
a truncated polynomial (definition \ref{trunc_poly-def}),
there would be  sub fractals of type $z^l + c$.  
\end{theorem}

Theorem \ref{theta-fractal} makes proof of this theorem 
a triviality. However, another proof can be furnished from the
perspective of scaling transform.
We demonstrate the proof here.
\begin{proof}[Proof of the  theorem \ref{domth} ]
We assume that:-
$$
l|g(z)| \le |f(z)| \le u|g(z)| \text{  with  } u,l > 0 
$$ 

If that is the case, then we know that the fractal:-
$$
z \leftarrow l g(z) + c 
$$
would be bigger in size than the fractal:-

$$
z \leftarrow u g(z) + c 
$$
if the fractal $z \leftarrow  g(z) + c$ exists due to 
the scaling theorem \eqref{scaling}.

That would mean that the fractal $z \leftarrow  f(z) + c$ 
is to be bound between the two fractals of the same shape,
which is only possible, iff the shape of the fractal 
$z \leftarrow  f(z) + c$ is similar to that of 
$z \leftarrow  g(z) + c$, which proves the theorem. 

\end{proof}

\end{subsection}

\subsection{Implications of Dominance}
The theorem \eqref{domth} suggests, that there will be Multibrots found in the vicinity 
of truncated polynomials. In this section we elaborate what observations we have made, and how 
those observations fit into the property of dominance. 
We demonstrate simple truncated polynomial fractals.

\subsubsection{Demonstration of Dominance}
To illustrate the result, we exhibit two transformation functions.
The first transformation function:-
$$
T_{2,3}(z) : z^2 + z^3 + c 
$$
has Mandelbrot ($z^2 + c$) fractals inside it. Figure  ~\ref{fig:3.x.1} on page ~\pageref{fig:3.x.1} shows the $T_{2,3}(z)$, and inner ($z^2 + c$) in the next image.

The seccond transformation function:-
$$
T_{4,7,-10}(z) : z^4 + z^7 - z^{10} + c 
$$
has Multibrot ($z^4 + c$) fractals inside it. Figure  ~\ref{fig:3.x.4} on page ~\pageref{fig:3.x.4} shows the $T_{4,7,-10}(z)$, and inner ($z^4 + c$) in the next image.

\begin{figure}
\begin{center}
\leavevmode
\includegraphics[scale=0.3]{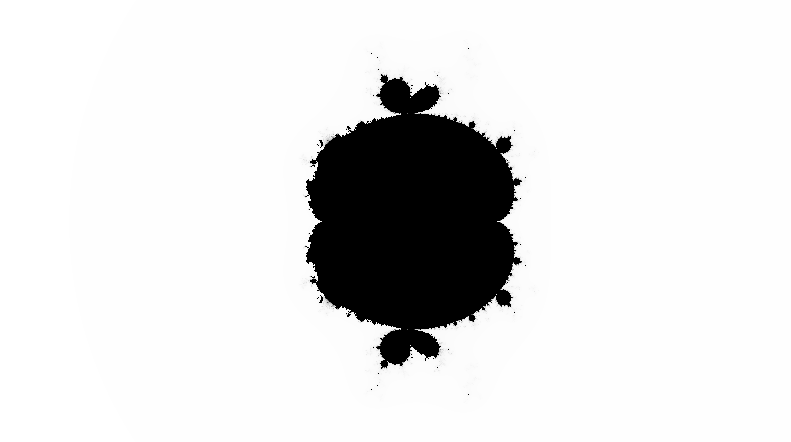}
\end{center}
\caption{The Polynomial Fractal of $z^2+z^3+c$ }
\label{fig:3.x.1}
\end{figure}

\begin{figure}
\begin{center}
\leavevmode
\includegraphics[scale=0.3]{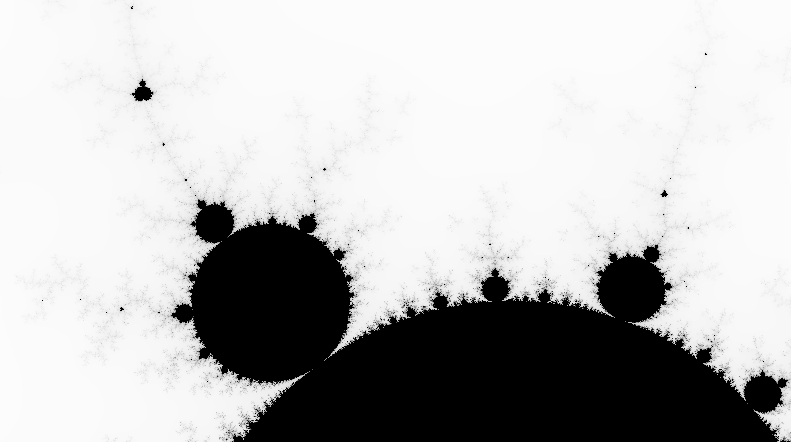}
\end{center}
\caption{ $z^2+z^3+c$ zoomed, showing existence of $z^2+c$ fractal inside }
\label{fig:3.x.2}
\end{figure}

\begin{figure}
\begin{center}
\leavevmode
\includegraphics[scale=0.3]{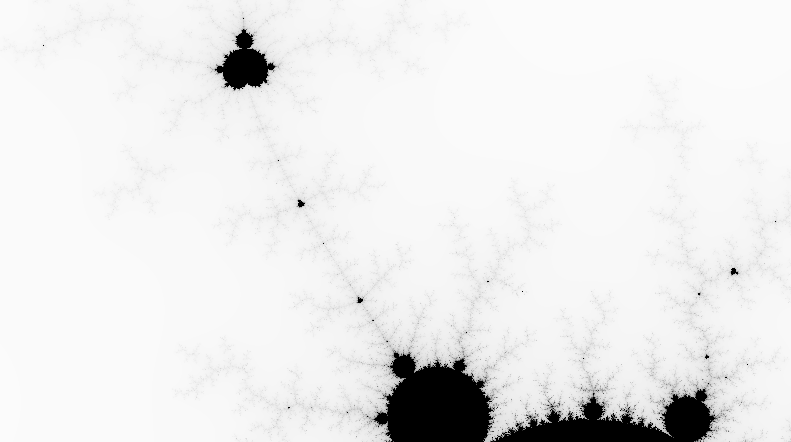}
\end{center}
\caption{ $z^2+z^3+c$, zoomed, showing clearly $z^2+c$ fractals}
\label{fig:3.x.3}
\end{figure}

\begin{figure}
\begin{center}
\leavevmode
\includegraphics[scale=0.3]{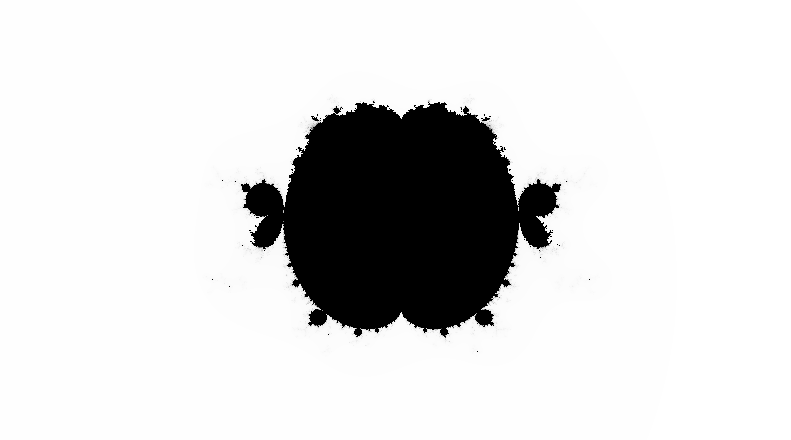}
\end{center}
\caption{ Image of ${(z-1)}^2 - {(z-1)}^3 + c$, showing the invariance under translation and rotation. 
Cross reference with Figure  ~\ref{fig:3.x.1} on page ~\pageref{fig:3.x.1}}
\label{fig:3.inv.1}
\end{figure}

\begin{figure}
\begin{center}
\leavevmode
\includegraphics[scale=0.3]{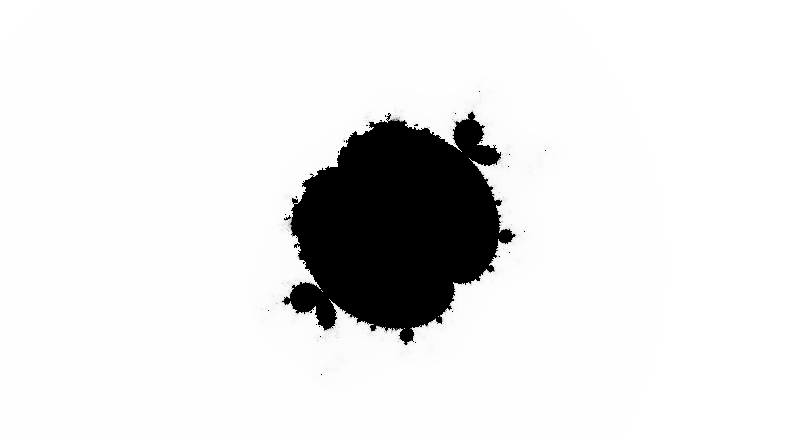}
\end{center}
\caption{ Image of $e^{(i\frac{\pi}{4})}{(z-1)}^2 + i{(z-1)}^3 + c $, 
Cross reference with Figure  ~\ref{fig:3.inv.1} on page ~\pageref{fig:3.inv.1}}
\label{fig:3.inv.2}
\end{figure}

\begin{figure}
\begin{center}
\leavevmode
\includegraphics[scale=0.3]{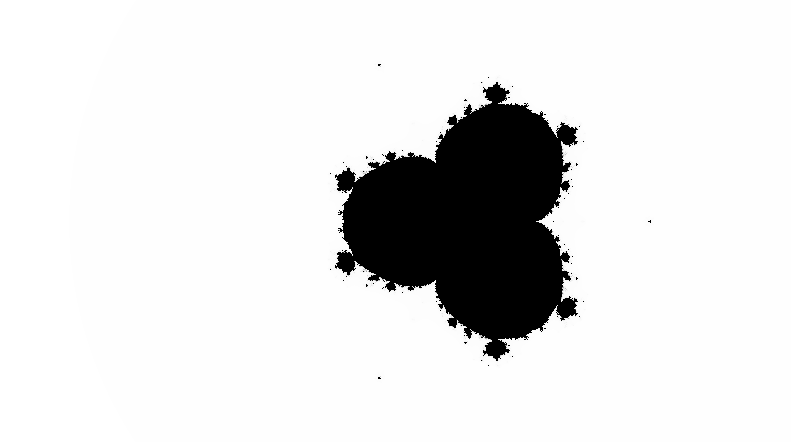}
\end{center}
\caption{The Polynomial Fractal of $z^4+z^7-z^{10}+c$, notice the similarity with figure : ~\ref{fig:1_1} on page ~\pageref{fig:1_1}}
\label{fig:3.x.4}
\end{figure}

\begin{figure}
\begin{center}
\leavevmode
\includegraphics[scale=0.3]{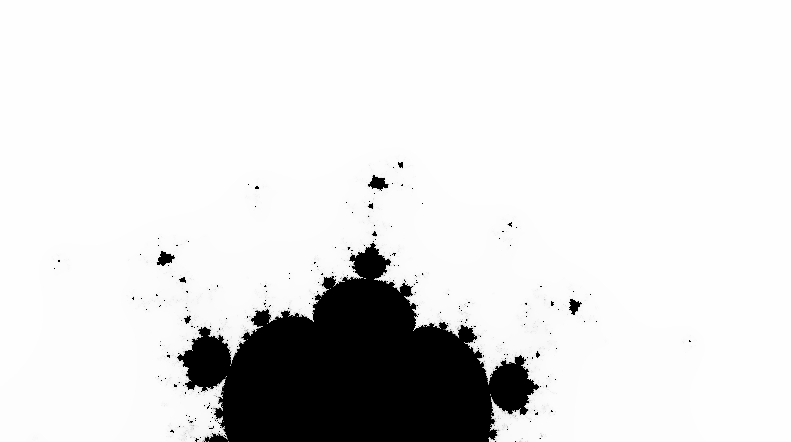}
\end{center}
\caption{The Polynomial Fractal of $z^4+z^7-z^{10}+c$, zoomed, showing how the borders look like a Multibrot with $n=4$}
\label{fig:3.x.5}
\end{figure}

\begin{figure}
\begin{center}
\leavevmode
\includegraphics[scale=0.3]{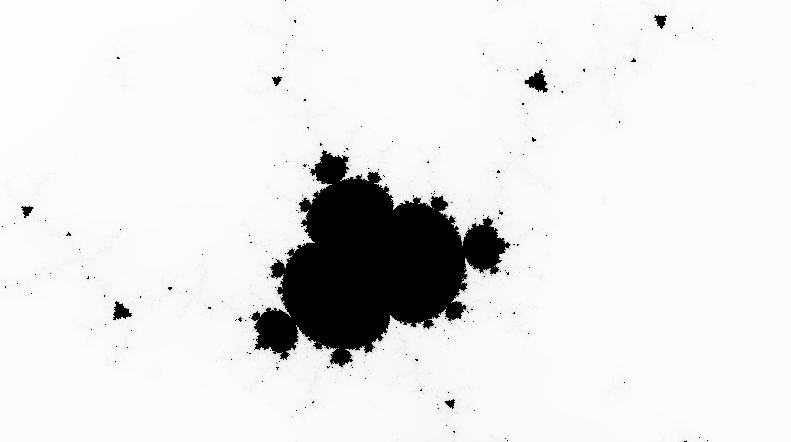}
\end{center}
\caption{ $z^4+z^7-z^{10}+c$ zoomed, showing $z^4+c$ fractals inside, compare with figure : ~\ref{fig:1_1} on page ~\pageref{fig:1_1} }
\label{fig:3.x.6}
\end{figure}

These two observations shows validity of the dominance property.

\subsubsection{Implication for Functions Expandable with Taylor Series}
The most important result of Dominance lies with the Taylor Series expansion.
For a function $f(z)$ which is capable of expansion by Taylor Series, ``multibrot'' type 
fractals can be generated by removing lower powers, or multiplying or dividing the function by $z^k$. 

We note down the expansion of $sin(z)$ :-
\begin{equation}\label{sin_z}
sin(z) = z - \frac{z^3}{3!} + \frac{z^5}{5!} - ...  
\end{equation}

and the expansion of $cos(z)$ :-

\begin{equation}\label{cos_z}
cos(z) = 1 - \frac{z^2}{2!} + \frac{z^4}{4!} - ...  
\end{equation}

In this subsection for brevity we have used the informal $G \sim F$ to denote $G \in \Theta_z(F)$.

Now clearly then due to dominance effect:-

\begin{enumerate}

\item{

$$
cos(z) - 1 + c \sim \frac{z^2}{2} + c \rightarrow z^2 + c 
$$
}

\item{
$$
6(sin(z) - z + c) \sim z^3 + c
$$
}

\item{
$$
sin(z^n) + c \sim z^n + c
$$
}

\item{
$$
z^k sin(z^n) + c \sim z^{n+k} + c  \; ; \; k \in \mathbb{R}
$$

}

\item{
And finally we explain, the initial $tan(z)$ fractals, which served as the starting point of this paper,in the next subsection.
}

\end{enumerate}

\subsection{ The $tan(z)$ Fractals}

Figure  ~\ref{fig:3__7} on page ~\pageref{fig:3__7} shows a ${(tan(z))}^2+c$ fractal, 
and figure ~\ref{fig:3_8} on page ~\pageref{fig:3_8} shows embedded Mandelbrots inside it.

\begin{figure}
\begin{center}
\leavevmode
\includegraphics[scale=0.3]{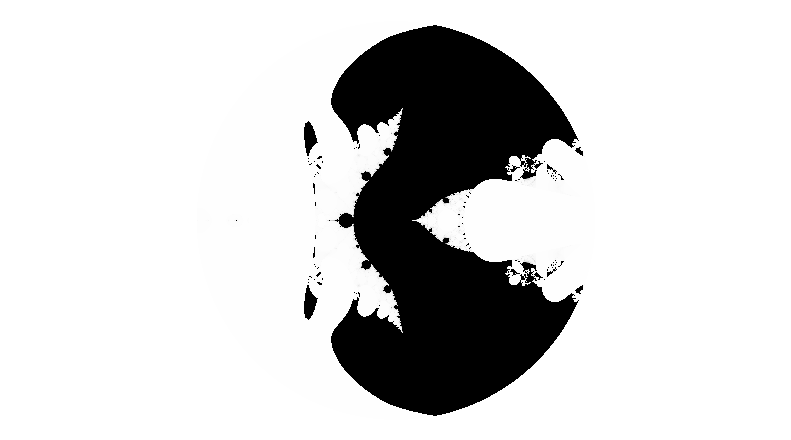}
\end{center}
\caption{ ${(tan(z))}^2+c$ Fractal }
\label{fig:3__7}
\end{figure}

\begin{figure}
\begin{center}
\leavevmode
\includegraphics[scale=0.3]{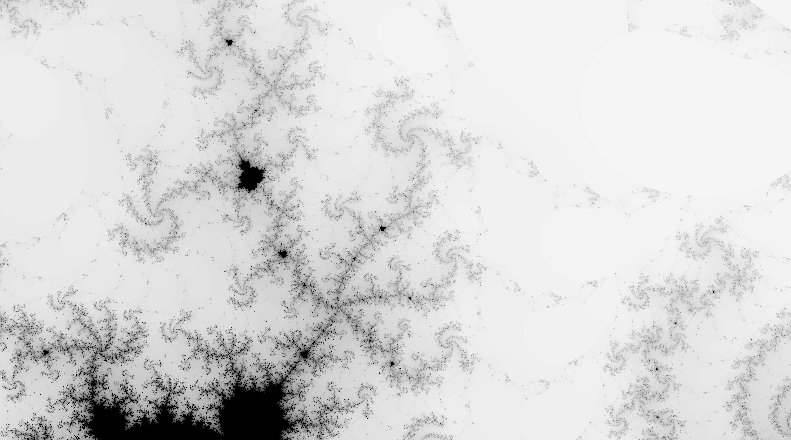}
\end{center}
\caption{ ${(tan(z))}^2+c$  zoomed, showing Mandelbrots inside }
\label{fig:3_8}
\end{figure}

\begin{equation}\label{tanz}
tan(z) = z + \frac{z^3}{3} + \frac{2 z^5}{15} + ... 
\end{equation}

hence, 
$$
{(tan(z))}^n + c \sim z^n + c
$$

Now we proceed to explain the observation of `${(tan(z))}^{-2} + c$'.
Figure  ~\ref{fig:3_9} on page ~\pageref{fig:3_9} shows a ${(cotan(z))}^2+c$ fractal, 
and figure ~\ref{fig:3_10} on page ~\pageref{fig:3_10} shows embedded Mandelbrots inside it.

\begin{figure}
\begin{center}
\leavevmode
\includegraphics[scale=0.3]{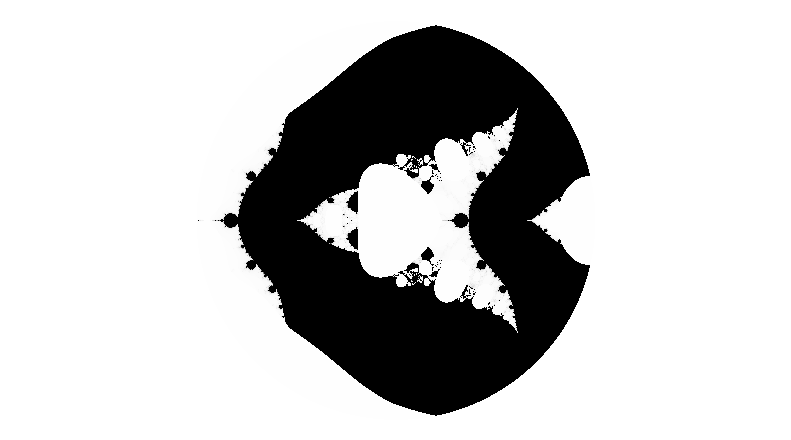}
\end{center}
\caption{ ${(cotan(z))}^2+c$ Fractal }
\label{fig:3_9}
\end{figure}

\begin{figure}
\begin{center}
\leavevmode
\includegraphics[scale=0.3]{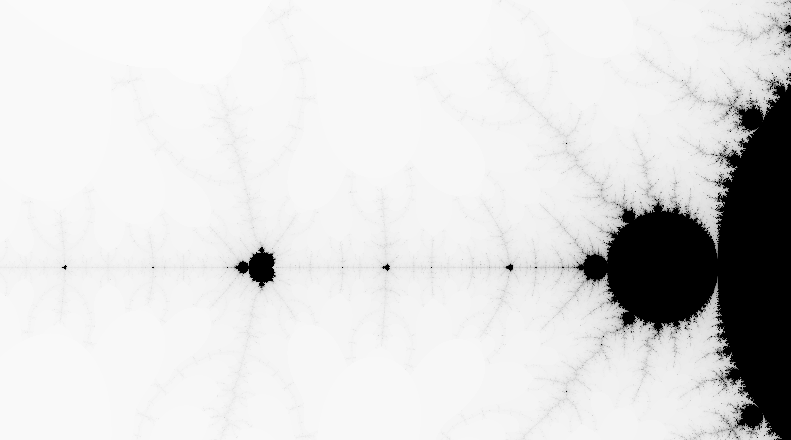}
\end{center}
\caption{ ${(cotan(z))}^2+c$  zoomed, showing Mandelbrots inside }
\label{fig:3_10}
\end{figure}

\begin{equation}\label{cotz}
\frac{1}{tan(z)} = cotan(z) = \frac{1}{z} - \frac{z}{3} - \frac{z^3}{45} - \frac{2 z^5}{945} - ...
\end{equation}

We note down that equation \eqref{cotz} contains power of $z^{-1}$, and hence, we can not directly use
the theorem of existence of embedded Multibrots. 

Let us define:-
\begin{equation}\label{z_inv}
g(a,b,z) = \frac{a}{z} - \frac{z}{b} \; ; \; a,b \in \mathbb{C}
\end{equation}

and 

\begin{equation}\label{z_invf}
T_g(a,b,n) :  z \leftarrow { (g(a,b,z)) }^n + c
\end{equation}

A fractal of type $T_g(1,3,2)$ is shown in figure  ~\ref{fig:3_9} on page ~\pageref{fig:3_9}.

\begin{figure}
\begin{center}
\leavevmode
\includegraphics[scale=0.3]{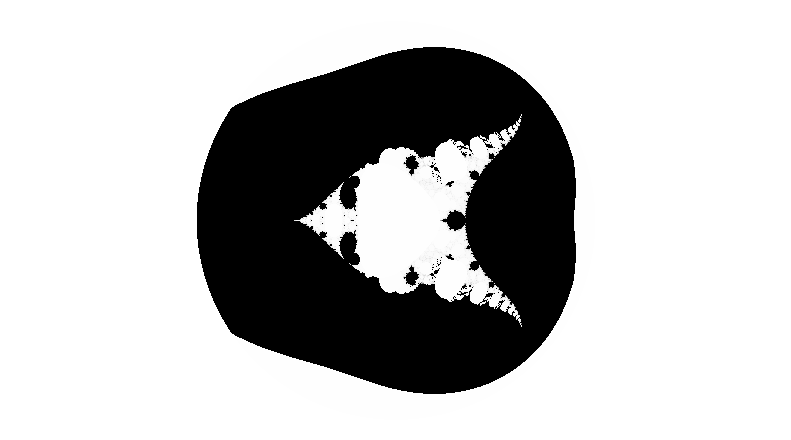}
\end{center}
\caption{ $T_g(1,3,2)$ Fractal, note the similarity with figure  ~\ref{fig:3_9} on page ~\pageref{fig:3_9} }
\label{fig:3_11}
\end{figure}

We state without proof:-

$$
{(cotan(z))}^n + c \sim   T_g(1,3,n) \; ; \; |z| < 1 
$$

and 

$$
{(cotan(z))}^n + c \sim   z^n + c \; ; \; |z| > 1 
$$

\begin{theorem}\label{zinv}
\textbf{Existence of  Multibrot inside Fractal of $T_g(a,b,n)$.}

Fractal of equation family $T_g(a,b,n)$ \eqref{z_invf} 
would contain Multibrots inside.
\end{theorem}

\begin{proof}[Proof of theorem \eqref{zinv}]
We note that when $z \to 0$, the equation \eqref{z_inv} gets unbounded. 

Assume $k$ is large enough so that:-

$$
k > a - \frac{1}{b}
$$

Now, when $z \to 1$ but $z \ne 1$ and $|z| < 1$, we have:-
\begin{equation}\label{eps}
z = 1 - \epsilon \; ; \; \epsilon \in \mathbb{C} \; ; \; \epsilon \to 0 \; ; \;  \epsilon \ne 0
\end{equation}

In that case:-
$$
\lim_{ z \to 1} g(a,b,z) = \lim_{ \epsilon \to 0} ( \frac{a}{1-\epsilon} - \frac{1-\epsilon}{b} ) 
$$ 

That approximates into :-
$$
\lim_{ \epsilon \to 0} ( a( 1 + \epsilon + {\epsilon}^2 + {\epsilon}^3 ...) - \frac{1-\epsilon}{b} ) =  a(1 + \epsilon) - \frac{1-\epsilon}{b} 
$$ 
Ignoring the higher powers of $\epsilon$ we get:-

\begin{equation}\label{z_inv_mult}
g(a,b,z) = 2a - z(a + \frac{1}{b}) \; ; \; z \to 1 \; ; \;  z \ne 1
\end{equation}

And hence equation \eqref{z_invf} can be written as:-
\begin{equation}\label{z_invf_sim}
T_g(a,b,n) :  z \leftarrow { (g(a,b,z)) }^n + c \sim {(2a - z(a + \frac{1}{b}))}^n + c
\end{equation}

But, equation \eqref{z_invf_sim} is a `shape preserving transform' of equation \eqref{spt}.
Hence, the fractal of equation \eqref{z_invf_sim} is expected to have Multibrot fractals inside.
\end{proof}
A zoom of fractal of type $T_g(2,5,3)$ is shown in figure  ~\ref{fig:3_12} on page ~\pageref{fig:3_12},
revealing the Multibrot with $n=3$ inside it.

\begin{figure}
\begin{center}
\leavevmode
\includegraphics[scale=0.3]{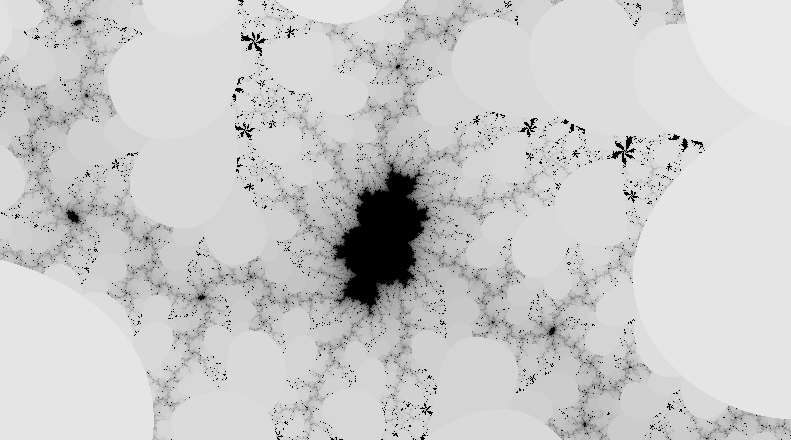}
\end{center}
\caption{ $T_g(2,5,3)$ Fractal, 1000 times zoomed, Shows a Multibrot inside with $n=$3 }
\label{fig:3_12}
\end{figure}

Hence we complete the explanations of the observations what we have stated 
in the beginning of the current section.

\end{section}

\begin{section}{Summary and Future Work and Thanks}
This paper started as a celebration of what Mandelbrot's work has been.
We started working on creating a software tool to produce `only' Mandelbrot set in color, and ended up
creating a software tool capable of evaluating any Mathematical expression, to generate escape time fractals,
in any color. While experimenting with this tool, we came up with heather to unknown, unexplained phenomenon.
The phenomenon of shape preserving transform, and the phenomenon of embedded inner fractal images,
were so fascinating that we tried to explain them, so as to generate a preliminary theoretical basis. 
The usage of transformation vector to preserve shape and transform any finite truncated polynomial fractal is proven. 
The dominance phenomenon we borrowed, and reused from Computer Science, 
and shown that it is a very powerful mechanism to predict existence of embedded inner fractals. 
While dominance principle shows promise, the principle is fairly new
at least in the fractal geometry domain. Further theoretical research work is needed to take the preliminary concepts 
discussed in this paper, into mainstream fractal geometry.

We would specially like to thank Carl Johansen, whose open sourced Mandelbrot viewer we borrowed, 
without which, not a single line of this paper could have been written.
\end{section}

\end{document}